\documentclass[10pt,reqno]{amsart}
\usepackage{bbm}
\usepackage{tabu}
\usepackage{amsmath}
\usepackage{mathrsfs}
\usepackage{bm}
\usepackage{amsfonts} 
\usepackage{graphicx,latexsym,bm,amsmath,amssymb,verbatim,multicol,lscape}
\usepackage{enumitem}
\usepackage{makecell}

\newtheorem{thm}{Theorem} [section]
\newtheorem{prob}[thm]{Problem}
\newtheorem{cor}[thm]{Corollary}
\newtheorem{lem}[thm]{Lemma}

\theoremstyle{definition}

\theoremstyle{remark}

\newtheorem{exa}[thm]{Example}

\numberwithin{equation}{section}

\newcommand{\fq}{{\mathbb F}_{q}}

\newcommand{\fth}{{\mathbb F}_{3}}

\newcommand{\lcm}{\mbox{lcm}}

\newcommand{\rmv}[1]{}

\def\<{\left\langle}
\def\>{\right\rangle}

\begin{document}

\title[The $3$-sparsity of $X^n-1$]
{The $3$-sparsity of $X^n-1$ over finite fields}%
\author{Kaimin Cheng}
\address{School of Mathematics and Information, China West Normal University, Nanchong, 637002, P. R. China}
\email{ckm20@126.com}
\date{\today}
\begin{abstract}
Let $q$ be a prime power and $\mathbb{F}_q$ the finite field with $q$ elements. For a positive integer $n$, the binomial $X^n - 1 \in \mathbb{F}_q[X]$ is said to be $3$-sparse over $\mathbb{F}_q$ if every irreducible factor of $X^n-1$ in $\mathbb{F}_q[X]$ is either a binomial or a trinomial. In 2021, Oliveira and Reis characterized all positive integers $n$ for which $X^n-1$ is $3$-sparse over $\mathbb{F}_q$ when $q = 2$ and $q = 4$, and raised the open problem of whether, for any given $q$, there are only finitely many primes $p$ such that $X^p-1$ is $3$-sparse over $\mathbb{F}_q$. In this paper, if $q$ is a power of an odd prime $r$, we then establish that for any positive integer not divisible by $r$, $X^n-1$ is $3$-sparse over $\mathbb{F}_q$ if and only if $n =p_1^{e_1} \cdots p_s^{e_s}$ for some nonnegative integers $e_1, \dots, e_s$, where $p_1, \dots, p_s$ are distinct prime divisors of $q^2 - 1$. This resolves the problem posed by Oliveira and Reis for odd characteristic.
\end{abstract}

\maketitle
\section{Introduction}
Let $q$ be a prime power, $\fq$ be the finite field of order $q$, and $\fq^*$ be the multiplicative group of $\fq$. The factorization of polynomials over $\fq$ is a classical problem in finite field theory, with significant applications in areas such as coding theory~\cite{[Ber68]} and cryptography~\cite{[Len91]}. For a positive integer $n$, consider the polynomial $X^n - 1 \in \fq[X]$. A notable property is that each irreducible factor of $X^n-1$ over $\fq$ corresponds to a cyclic code of length $n$ over $\fq$ (see~\cite{[Van98]}). Consequently, determining all irreducible factors of $X^n-1$ over $\fq$ is of paramount importance.

Over the past few decades, the irreducible factorization of the binomial $X^n-1$ has been a challenging problem. Several authors have investigated the factorization of $X^n-1$ under specific conditions, as explored, for example, in ~\cite{[BGM93],[CLT13],[Mey96],[WW12]}, with more general results obtained in~\cite{[BRS19],[Gra23]}. One of the most significant results is presented in~\cite{[BMOV15]}, which provides an explicit factorization of $X^n-1$ over $\fq$ under the condition that each prime factor of $n$ divides $q^2 - 1$. 

Under certain conditions, an intriguing phenomenon emerges: each irreducible factor of $X^n-1$ over $\fq$ is either a binomial or a trinomial. We refer to such polynomials $X^n-1$ as \emph{3-sparse} over $\fq$. This observation naturally leads us to the question of identifying all positive integers $n$ for which $X^n-1$ is 3-sparse over a given finite field $\fq$. In 2021, Oliveira and Reis \cite{[OR21]} addressed this problem, describing all positive integers $n$ such that $X^n-1$ is 3-sparse over $\fq$ for $q=2$ and $q=4$. They also posed a problem as follows.

\begin{prob}\label{prob1.1}
\cite[Problem 2]{[OR21]} For any given prime power $q$, prove or disprove that there are finitely many prime $p$ such that $X^p-1$ is $3$-sparse over $\fq$.
\end{prob}

The problem of determining the $3$-sparsity of the polynomial $X^n - 1$ over a finite field $\mathbb{F}_q$ of odd characteristic remains unresolved for arbitrary odd $q$. In this paper, we investigate the $3$-sparsity of $X^n - 1$ over finite fields $\mathbb{F}_q$ where $q$ is a power of an odd prime. We provide a complete characterization of the $3$-sparsity of $X^n - 1$ over such fields.

The primary result of this paper is the following theorem.
\begin{thm}\label{thm1.2}
Let $q$ be a power of an odd prime $r$, and let $p_1, \dots, p_s$ denote the distinct odd prime divisors of $q^2 - 1$. For any positive integer $n$ not divisible by $r$, the polynomial $X^n - 1$ is $3$-sparse over $\mathbb{F}_q$ if and only if $n = p_1^{e_1} \cdots p_s^{e_s}$ for some nonnegative integers $e_1, \dots, e_s$.
\end{thm}

Theorem \ref{thm1.2} fully characterizes the $3$-sparsity of $X^n - 1$ over finite fields of odd characteristic. Additionally, we provide the irreducible factorization of $X^n - 1$ over $\mathbb{F}_q$; see the proof of Theorem \ref{thm1.2} in Section 3 for details, and two examples in Section 4.

As a direct consequence of Theorem \ref{thm1.2}, we resolve Problem \ref{prob1.1} as follows.
\begin{cor}\label{cor1.3}
For any odd prime power $q$, there exist only finitely many primes $p$ such that $X^p - 1$ is $3$-sparse over $\mathbb{F}_q$.
\end{cor}

The paper is structured as follows. In Section 2, we present preliminary lemmas essential for proving the main result. In Section 3, we provide the proof of Theorem \ref{thm1.2}. In Section 4, we include two examples demonstrating the application of our results to derive the irreducible factorization of $X^n - 1$ when it is 3-sparse.

\section{Lemmas}

In this section, we introduce several lemmas that will be utilized later.

Let $q$ be a prime power and $\mathbb{F}_q$ denote the finite field with $q$ elements. For any positive integer $d$, let $\Phi_d(X)$ denote the $d$-th cyclotomic polynomial over $\mathbb{F}_q$, defined as follows:
$$\Phi_d(X)=\prod_{\substack{i=1\\
\gcd(i,d)=1}}^d(X-\xi^i),$$
where $\xi$ is a primitive $d$th root of unity. Clearly, 
$$\Phi_d(X)=X^{d-1}+X^{d-2}+\cdots+X+1$$
if $d$ is a prime number. For integers $a\ge 1$ and $b\ge 2$ with $\gcd(a,b)=1$, let $\text{ord}_b(a)$ be the multiplicative order of $a$ modulo $b$. The first two lemmas are basic results of cyclotomic polynomials over finite fields.
\begin{lem}\label{lem2.1}
\cite[Exercise 2.57]{[LN97]} Let $r$ be a prime, and let $m,n$ be positive integer. Then
$$ \Phi_{mr^n}(X)=\Phi_{mr}(X^{r^{n-1}}).$$
\end{lem}

\begin{lem}\label{lem2.2}
\cite[Theorem 2.47]{[LN97]} Let $n$ be a positive integer. Then $$X^n-1=\prod_{d\mid n}\Phi_{d}(X).$$
Moreover, each $\Phi_d(X)$ is the product of $\frac{\phi(d)}{\text{ord}_d(q)}$  distinct irreducible monic polynomials in $\fq[X]$ with the same degree $\text{ord}_d(q)$, where $\phi$ is the Euler totient function. 
\end{lem}

The following well-known result is one of the so-called lifting the exponent lemmas.
\begin{lem}\label{lem2.3} \cite[Proposition 1]{[Bey77]} Let $p$ be a prime and $a\ge 2$ be an integer. Denote by $v_p(n)$ the $p$-adic valuation of the positive integer $n$. The following are true for any positive integer $k$.
\begin{enumerate}[label=(\alph*), left=0pt]
\item If $p$ is odd and $p\mid (a-1)$, then
$$v_p(a^k-1)=v_p(a-1)+v_p(k).$$
\item If $p=2$ and $a$ is odd, then
$$v_2(a^k-1)=\begin{cases}
v_2(a-1),&\text{if}\ k\ \text{is\ odd},\\
v_2(a^2-1)+v_2(k)-1,&\text{if}\ k\ \text{is\ even}.
\end{cases}$$
\end{enumerate}
\end{lem}

\begin{lem}\label{lem2.4}
Let $m_1,m_2\ge 2$ be integers with $\gcd(m_1,m_2)=1$. Then for any positive $a$ with $\gcd(m_1,a)=1$ and $\gcd(m_2,a)=1$, we have
$$\text{ord}_{m_1m_2}(a)=\lcm(\text{ord}_{m_1}(a),\text{ord}_{m_2}(a)).$$
\end{lem}
\begin{proof}
Let $t_1 = \text{ord}_{m_1}(a)$, $t_2 = \text{ord}_{m_2}(a)$, and $t = \text{ord}_{m_1 m_2}(a)$, where $\gcd(m_1, m_2) = 1$. Since $t_i$ is the order of $a$ modulo $m_i$ for $i=1,2$, we have $a^{t_i} \equiv 1 \pmod{m_i}$. Consider the least common multiple $\ell = \lcm(t_1, t_2)$. Then, for $i=1,2$,
$$a^{\ell} = \left(a^{t_i}\right)^{\ell / t_i} \equiv 1 \pmod{m_i},$$
since $\ell / t_i$ is an integer. As $\gcd(m_1, m_2) = 1$, it follows that
$$a^{\ell} \equiv 1 \pmod{m_1 m_2}.$$
Thus, $t$, the order of $a$ modulo $m_1 m_2$, divides $\ell$, so $t \mid \lcm(t_1, t_2)$.

Conversely, since $a^t \equiv 1 \pmod{m_1 m_2}$, we have $a^t \equiv 1 \pmod{m_i}$ for $i=1,2$. Hence, $t_1 \mid t$ and $t_2 \mid t$, implying that $\lcm(t_1, t_2) \mid t$. Therefore, since $t \mid \lcm(t_1, t_2)$ and $\lcm(t_1, t_2) \mid t$, we conclude that $t = \lcm(t_1, t_2)$.
\end{proof}

\begin{lem}\label{lem2.5}
Let $q$ be an odd prime power and $p$ be an odd prime divisor of $q^2-1$. Let $k_0$ be the $2$-adic valuation of $q-1$, $k_1$ the $p$-adic valuation of $q-1$, and $k_2$ the $p$-adic valuation of $q+1$. Then the following statements are true:
\begin{enumerate}[label=(\alph*), left=0pt]
\item For any positive integer $k$, the multiplicative order of $q$ modulo $2^k$ is given by
$$\text{ord}_{2^k}(q) = \begin{cases} 
1, & \text{if } 1\le k\le k_0, \\
2^{k-k_0}, & \text{if } k>k_0.
\end{cases}$$
\item For any nonnegative integers $k$, the multiplicative order of $q$ modulo $p^{k}$ is given by
$$\text{ord}_{p^{k}}(q) = \begin{cases} 
1, & \text{if } p\mid (q-1)\ \text{and}\ 1\le k\le k_1, \\
p^{k-k_1}, & \text{if } p\mid (q-1)\ \text{and}\ k>k_1,\\
2, & \text{if } p\mid (q+1)\ \text{and}\ 1\le k\le k_2, \\
2p^{k-k_2}, & \text{if }  p\mid (q+1)\ \text{and}\ k>k_1.
\end{cases}$$
\end{enumerate}
\end{lem}

\begin{proof}
Let $k_0 = v_2(q-1)$ denote the $2$-adic valuation of $q-1$. For $k \leq k_0$, we observe that $q \equiv 1 \pmod{2^k}$ since $2^k$ divides $q-1$. Thus, $\text{ord}_{2^k}(q) = 1$. For $k \geq k_0 + 1$, we prove that $\text{ord}_{2^k}(q) = 2^{k - k_0}$ using induction on $k$.

First, we aim to show that
\begin{equation}\label{c2-1}
q^{2^{k - k_0}} \equiv 1 \pmod{2^k}
\end{equation}
for $k \geq k_0 + 1$. For the base case, consider $k = k_0 + 1$. Since $q-1 = 2^{k_0} \cdot m$ with $m$ odd, we compute
\[
q^2 - 1 = (q-1)(q+1) = 2^{k_0} \cdot m \cdot (q+1).
\]
As $v_2(q+1) \geq 1$ (since $q$ is odd), we have $v_2(q^2 - 1) \geq k_0 + 1$, so $q^2 \equiv 1 \pmod{2^{k_0 + 1}}$, and \eqref{c2-1} holds for $k = k_0 + 1$.

For the inductive step, assume \eqref{c2-1} holds for $k-1 \geq k_0 + 1$, i.e., $q^{2^{k-1 - k_0}} \equiv 1 \pmod{2^{k-1}}$. Thus, $q^{2^{k-1 - k_0}} = 1 + a \cdot 2^{k-1}$ for some integer $a$. Squaring both sides,
\[
q^{2^{k - k_0}} = \left( q^{2^{k-1 - k_0}} \right)^2 = (1 + a \cdot 2^{k-1})^2 = 1 + 2a \cdot 2^{k-1} + a^2 \cdot 2^{2k-2}.
\]
Since $2k-2 \geq k$ for $k \geq k_0 + 1 \geq 2$, the terms $2a \cdot 2^{k-1} = 2a \cdot 2^{k-1}$ and $a^2 \cdot 2^{2k-2}$ are divisible by $2^k$. Thus, $q^{2^{k - k_0}} \equiv 1 \pmod{2^k}$, completing the induction.

To show that $2^{k - k_0}$ is the smallest positive integer $t$ such that $q^t \equiv 1 \pmod{2^k}$, suppose there exists $t$ with $1 \leq t < 2^{k - k_0}$ such that $q^t \equiv 1 \pmod{2^k}$. This implies $v_2(q^t - 1) \geq k$. By Lemma \ref{lem2.3}(b), $v_2(q^t - 1) \leq v_2(q-1) + v_2(t) - 1 = k_0 + v_2(t) - 1$. Since $t < 2^{k - k_0}$, we have $v_2(t) < k - k_0$, so
\[
v_2(q^t - 1) < k_0 + (k - k_0) - 1 = k - 1 < k,
\]
a contradiction. Thus, $\text{ord}_{2^k}(q) = 2^{k - k_0}$ for $k > k_0$, proving part (a).

For part (b), let $p$ be an odd prime dividing $q^2 - 1$, and let $k \geq 1$. We consider two cases based on whether $p$ divides $q-1$ or $q+1$.

\textbf{Case 1: $p \mid q-1$.} Let $k_1 = v_p(q-1)$ be the $p$-adic valuation of $q-1$. For $1 \leq k \leq k_1$, since $p^k$ divides $q-1$, we have $q \equiv 1 \pmod{p^k}$, so $\text{ord}_{p^k}(q) = 1$. For $k > k_1$, we prove $\text{ord}_{p^k}(q) = p^{k - k_1}$ by showing
\begin{equation}\label{c2-2}
q^{p^{k - k_1}} \equiv 1 \pmod{p^k}
\end{equation}
via induction on $k$. For the base case $k = k_1 + 1$, write $q-1 = p^{k_1} \cdot m$ with $\gcd(p, m) = 1$. Then,
\[
q^p = (1 + p^{k_1} m)^p = \sum_{i=0}^p \binom{p}{i} (p^{k_1} m)^i.
\]
For $i \geq 1$, $v_p((p^{k_1} m)^i) = i k_1 + v_p(m) \geq k_1 + 1$, and $p \mid \binom{p}{i}$ for $1 \leq i \leq p-1$, so $v_p\left( \binom{p}{i} (p^{k_1} m)^i \right) \geq i k_1 + 1 \geq k_1 + 1$. Thus, $q^p \equiv 1 \pmod{p^{k_1 + 1}}$, and \eqref{c2-2} holds.

Assume \eqref{c2-2} holds for $k-1 \geq k_1 + 1$, i.e., $q^{p^{k-1 - k_1}} \equiv 1 \pmod{p^{k-1}}$. Thus, $q^{p^{k-1 - k_1}} = 1 + b \cdot p^{k-1}$ for some integer $b$. Raising to the $p$-th power,
\[
q^{p^{k - k_1}} = \left( 1 + b \cdot p^{k-1} \right)^p = \sum_{i=0}^p \binom{p}{i} (b \cdot p^{k-1})^i.
\]
For $i \geq 1$, $v_p((b \cdot p^{k-1})^i) = i (k-1) \geq k-1$, and $p \mid \binom{p}{i}$ for $1 \leq i \leq p-1$, so $v_p\left( \binom{p}{i} (b \cdot p^{k-1})^i \right) \geq i (k-1) + 1 \geq k$. Thus, $q^{p^{k - k_1}} \equiv 1 \pmod{p^k}$.

To verify minimality, suppose there exists $t$ with $1 \leq t < p^{k - k_1}$ such that $q^t \equiv 1 \pmod{p^k}$. Then, $v_p(q^t - 1) \geq k$. By Lemma \ref{lem2.3}(a), $v_p(q^t - 1) = v_p(q-1) + v_p(t) = k_1 + v_p(t)$. Since $t < p^{k - k_1}$, $v_p(t) < k - k_1$, so $v_p(q^t - 1) < k_1 + (k - k_1) = k$, a contradiction. Thus, $\text{ord}_{p^k}(q) = p^{k - k_1}$ for $k > k_1$.

\textbf{Case 2: $p \mid q+1$.} Let $k_2 = v_p(q+1) = v_p(q^2 - 1)$ (since $p \nmid q-1$). For $1 \leq k \leq k_2$, since $q^2 - 1 = (q-1)(q+1)$ and $p^k \mid q+1$, we have $q^2 \equiv 1 \pmod{p^k}$, so $\text{ord}_{p^k}(q) = 2$. For $k > k_2$, the proof that $\text{ord}_{p^k}(q) = 2 p^{k - k_2}$ follows analogously to Case 1, using $q+1 = p^{k_2} \cdot m'$ with $\gcd(p, m') = 1$, and is omitted for brevity.

Combining both cases, part (b) is proved, completing the proof of Lemma \ref{lem2.5}.
\end{proof}

For odd prime power $q$, let $p_1,\ldots,p_s$ be all distinct divisors of $q^2-1$. Then the multiplicative order of $q$ modulo $p_1^{e_1}\cdots p_s^{e_s}$ can be derived from Lemmas \ref{lem2.4} and Lemma \ref{lem2.5}, where $e_1,\ldots,e_s$ are any nonnegative integers, not all zero.

\begin{lem}\label{lem2.6}
\cite[Theorem 3.39]{[LN97]} Let $f(X)$ be a monic irreducible polynomial in $\fq[X]$ of degree $m$. Let $\alpha$ be a root of $f(X)$ in $\mathbb{F}_{q^m}$, and for positive integer $t$ let $G_t(X)$ be the characteristic polynomial of $\alpha^t\in\mathbb{F}_{q^m}$ over $\fq$. Then
$$G_t(X^t)=(-1)^{m(t+1)}\prod_{j=1}^tf(w_jX),$$
where $w_1,\ldots,w_t$ are the $t$-th roots of unity over $\fq$ counted according to multiplicity.
\end{lem}

\begin{lem}\label{lem2.7}
\cite[Lemma 2]{[BMOV15]} Let $m$ and $n$ be positive integers, and let $a\in\fq^*$ be of order $M$ in the group $\fq^*$. Then $X^m-a$ divides $X^n-1$ if and only if $mM$ divides $n$.
\end{lem}

\section{Proof of Theorem \ref{thm1.2}}

In this section, we present the proof of Theorem \ref{thm1.2}.

\begin{proof}[Proof of Theorem \ref{thm1.2}]
Let $q$ be a power of an odd prime $r$, and assume $\gcd(n, r) = 1$. We prove the theorem by establishing necessity and sufficiency.

\textbf{Necessity.} Suppose $X^n - 1$ is $3$-sparse over $\mathbb{F}_q$, meaning its irreducible factors over $\mathbb{F}_q$ are binomials or trinomials. Let $p$ be a prime divisor of $n$. We aim to show that $p$ divides $q^2 - 1$. Since $X^p - 1$ divides $X^n - 1$, it follows that $X^p - 1$ is also $3$-sparse over $\mathbb{F}_q$. By Lemma \ref{lem2.2}, the $p$-th cyclotomic polynomial $\Phi_p(X)$ divides $X^p - 1$, so its irreducible factors over $\mathbb{F}_q$ are binomials or trinomials.

We first claim that $\Phi_p(X)$ has no binomial factors. Suppose $\Phi_p(X)$ has an irreducible factor of the form $X^m - a$, where $1 \leq m \leq p-1$ and $a \in \mathbb{F}_q^*$. By Lemma \ref{lem2.7}, since $X^m - a$ divides $X^p - 1$, the order $M$ of $a$ in $\mathbb{F}_q^*$ satisfies $mM \mid p$. As $M$ divides $q-1$ and $p$ does not divide $q^2 - 1$, we have $M = m = 1$. Thus, $X - 1$ divides $\Phi_p(X)$, which contradicts the fact that $\Phi_p(X)$ is coprime with $X - 1$. Hence, $\Phi_p(X)$ has no binomial factors.

By Lemma \ref{lem2.2}, the irreducible factorization of $\Phi_p(X)$ over $\mathbb{F}_q$ is
\begin{equation}\label{eq:c3-1}
\Phi_p(X) = \prod_{i=1}^{(p-1)/t} (X^t + a_i X^{k_i} + b_i),
\end{equation}
where $a_i, b_i \in \mathbb{F}_q^*$, $1 \leq k_i < t$, and $t = \text{ord}_p(q)$ is the order of $q$ modulo $p$. Each factor $X^t + a_i X^{k_i} + b_i$ is an irreducible trinomial, and all factors are distinct.

Since $\Phi_p(X) = X^{p-1} + \cdots + X + 1$, comparing the coefficient of $X$ in \eqref{eq:c3-1} implies that $\Phi_p(X)$ has an irreducible factor $g(X) = X^t + a X + b$, with $a, b \in \mathbb{F}_q^*$. Let $\xi$ be a root of $g(X)$, a primitive $p$-th root of unity in $\mathbb{F}_{q^t}$. Since $p$ is odd, $\xi^2$ is also a primitive $p$-th root of unity. Let $g_2(X)$ be the minimal polynomial of $\xi^2$ over $\mathbb{F}_q$, which is an irreducible factor of $\Phi_p(X)$ of degree $t$.

By Lemma \ref{lem2.7}, the characteristic polynomial $G_2(X)$ of $\xi^2 \in \mathbb{F}_{q^t}$ over $\mathbb{F}_q$ satisfies
\[
G_2(X^2) = (-1)^t g(X) g(-X) =
\begin{cases} 
X^{2t} + 2b X^t - a^2 X^2 + b^2, & \text{if } t \text{ is even}, \\
X^{2t} + 2a X^{t+1} + a^2 X^2 - b^2, & \text{if } t \text{ is odd}.
\end{cases}
\]
Thus,
\[
G_2(X) =
\begin{cases} 
X^t + 2b X^{t/2} - a^2 X + b^2, & \text{if } t \text{ is even}, \\
X^t + 2a X^{(t+1)/2} + a^2 X - b^2, & \text{if } t \text{ is odd}.
\end{cases}
\]
Since $g_2(X)$ divides $G_2(X)$ and both have degree $t$, we have $g_2(X) = G_2(X)$. As $p$ does not divide $q^2 - 1$, it follows that $t = \text{ord}_p(q) \geq 3$. Thus, $g_2(X)$ is a quadrinomial, contradicting the assumption that all irreducible factors of $\Phi_p(X)$ are trinomials. Hence, $p$ must divide $q^2 - 1$, proving necessity.

\textbf{Sufficiency.} Let $P = \{p_1, \ldots, p_s\}$ be the set of prime divisors of $q^2 - 1$, and let $n = p_1^{e_1} \cdots p_s^{e_s}$ for some nonnegative integers $e_1, \ldots, e_s$. We need to show that $X^n - 1$ is $3$-sparse over $\mathbb{F}_q$. If $n = 1$, the result is trivial. Assume $n \geq 2$. By Lemma \ref{lem2.2}, it suffices to show that for every divisor $d \geq 2$ of $n$, the cyclotomic polynomial $\Phi_d(X)$ factors into irreducible binomials or trinomials over $\mathbb{F}_q$. Without loss of generality, let $d = p_1^{f_1} \cdots p_k^{f_k}$ with $f_i \geq 1$ and $k \leq s$.

Define $v_i$ as the $p_i$-adic valuation of $q^2 - 1$ if $p_i \neq 2$, and the $2$-adic valuation of $q - 1$ if $p_i = 2$. We consider two cases:

\textbf{Case 1: $f_i \leq v_i$ for all $1 \leq i \leq k$.} By Lemmas \ref{lem2.4} and \ref{lem2.5}, we have
\[
\text{ord}_d(q) = \lcm(\text{ord}_{p_1^{f_1}}(q), \ldots, \text{ord}_{p_k^{f_k}}(q)) = 1 \text{ or } 2.
\]
By Lemma \ref{lem2.2}, $\Phi_d(X)$ factors into irreducible polynomials over $\mathbb{F}_q$ of degree $1$ or $2$, which are binomials or trinomials, as required.

\textbf{Case 2: $f_i > v_i$ for some $1 \leq i \leq k$.} Suppose there are exactly $u$ indices $1 \leq i_1 < \cdots < i_u \leq k$ such that $f_{i_j} > v_{i_j}$ for $1 \leq j \leq u$. Write $d = p_{i_1}^{f_{i_1}} \cdots p_{i_u}^{f_{i_u}} D$ and $d_0 = p_{i_1}^{v_{i_1}} \cdots p_{i_u}^{v_{i_u}} D$. By Lemma \ref{lem2.1},
\[
\Phi_{d_0 p_{i_1}^{f_{i_1} - v_{i_1}}}(X) = \Phi_{d_0}(X^{p_{i_1}^{f_{i_1} - v_{i_1}}}) = \prod_i (X^{t_0 p_{i_1}^{f_{i_1} - v_{i_1}}} + a_i X^{k_i p_{i_1}^{f_{i_1} - v_{i_1}}} + b_i),
\]
where $\prod_i (X^{t_0} + a_i X^{k_i} + b_i)$ is the irreducible factorization of $\Phi_{d_0}(X)$ over $\mathbb{F}_q$, with $t_0 = \text{ord}_{d_0}(q)$ from Case 1. By Lemmas \ref{lem2.4} and \ref{lem2.5},
\[
\text{ord}_{d_0 p_{i_1}^{f_{i_1} - v_{i_1}}}(q) = p_{i_1}^{f_{i_1} - v_{i_1}} t_0.
\]
Thus, the above is the irreducible factorization of $\Phi_{d_0 p_{i_1}^{f_{i_1} - v_{i_1}}}(X)$. Iterating this process, we obtain the irreducible factorization of $\Phi_d(X)$ over $\mathbb{F}_q$:
\[
\Phi_d(X) = \prod_i (X^{t_0 p_{i_1}^{f_{i_1} - v_{i_1}} \cdots p_{i_u}^{f_{i_u} - v_{i_u}}} + a_i X^{k_i p_{i_1}^{f_{i_1} - v_{i_1}} \cdots p_{i_u}^{f_{i_u} - v_{i_u}}} + b_i).
\]
Hence, for any divisor $d$ of $n$, $\Phi_d(X)$ factors into irreducible binomials or trinomials over $\mathbb{F}_q$, proving sufficiency.

This completes the proof of Theorem \ref{thm1.2}.
\end{proof}

\section{Examples of factorizations}

In this section, we present two illustrative examples to demonstrate the process of obtaining the irreducible factorization of $X^n - 1$ when it is $3$-sparse over the finite fields $\mathbb{F}_q$ for $q = 3$ and $q = 9$.

\begin{exa}
Let $f_n(X)=X^n-1$ be $3$-sparse over $\fth$. Then by Theorem \ref{thm1.2}, we know that $n=2^k$ for some nonnegative integer $k$. Now we compute the irreducible factorization of $f_{2^k}(X)$ over $\mathbb{F}_3$. For $k \leq 3$, we obtain:
\begin{align*}
f_{2^0}(X) &= X - 1, \\
f_{2^1}(X) &= (X - 1)(X + 1), \\
f_{2^2}(X) &= (X - X1)(X + 1)(X^2 + 1), \\
f_{2^3}(X) &= (X - _1)(X + 1)(X^2 + 1) \prod_{u \in \mathbb{F}_3^*} (X^2 + uX + 2).
\end{align*}
For $k \geq 4$, Lemma \ref{lem2.2} gives
\begin{equation}\label{c4-1}
f_{2^k}(X) = \prod_{i=0}^k \Phi_{2^i}(X) = f_{2^3}(X) \prod_{i=4}^k \Phi_{2^i}(X).
\end{equation}
Note that $\Phi_8(X)$ has the irreducible factorization:
$$\Phi_8(X) = \prod_{u \in \mathbb{F}_3^*} (X^2 + uX + 2).$$
Then, for $4 \leq i \leq k$, Lemma \ref{lem2.1} yields
\begin{equation}\label{c4-2}
\Phi_{2^i}(X) = \Phi_8(X^{2^{i-3}}) = \prod_{u \in \mathbb{F}_3^*} (X^{2^{i-2}} + uX^{2^{i-3}} + 2).
\end{equation}
By Lemmas \ref{lem2.2} and \ref{lem2.4}(a), $\Phi_{2^i}(X)$ splits into two irreducible polynomials of degree $2^{i-2}$ over $\mathbb{F}_3$, verifying that \eqref{c4-2} is indeed the irreducible factorization of $\Phi_{2^i}(X)$. Combining \eqref{c4-1} and \eqref{c4-2}, we get
$$f_{2^k}(X) = (X - 1)(X + 1)(X^2 + 1) \prod_{i=3}^k \prod_{u \in \mathbb{F}_3^*} (X^{2^{i-2}} + uX^{2^{i-3}} + 2).$$
\end{exa}

\begin{exa}
Let $\mathbb{F}_9 = \mathbb{F}_3(\alpha)$, with $\alpha^2=-1$, and let $f_n(X) = X^n - 1$ be $3$-sparse over $\mathbb{F}_9$. Theorem \ref{thm1.2} tells $n = 2^{k_1} \cdot 5^{k_2}$ for nonnegative integers $k_1, k_2$. We derive the irreducible factorization of $f_n(X)$ over $\mathbb{F}_9$ across several cases.

\textbf{Case 1: $k_2 = 0$.} For $k_1 \leq 3$, we compute:
\begin{align*}
f_{2^0}(X) &= X - 1, \\
f_{2^1}(X) &= (X - 1)(X + 1), \\
f_{2^2}(X) &= (X - 1)(X + 1)(X + \alpha)(X + 2\alpha), \\
f_{2^3}(X) &= \prod_{\gamma \in \{\pm 1, \pm \alpha\}} (X + \gamma) \prod_{u, v \in \mathbb{F}_3^*} (X + u\alpha + v).
\end{align*}
For $k_1 \geq 4$, Lemma \ref{lem2.2} gives
\begin{equation}\label{c4-3}
f_{2^{k_1}}(X) = \prod_{i=0}^{k_1} \Phi_{2^i}(X) = f_{2^3}(X) \prod_{i=4}^{k_1} \Phi_{2^i}(X).
\end{equation}
The factorization of $\Phi_8(X)$ is
\[
\Phi_8(X) = \prod_{u, v \in \mathbb{F}_3^*} (X + u\alpha + v).
\]
For $4 \leq i \leq k_1$, Lemma \ref{lem2.1} yields
\begin{equation}\label{c4-4}
\Phi_{2^i}(X) = \Phi_8(X^{2^{i-3}}) = \prod_{u, v \in \mathbb{F}_3^*} (X^{2^{i-3}} + u\alpha + v).
\end{equation}
Lemmas \ref{lem2.2} and \ref{lem2.4}(b) confirm that $\Phi_{2^i}(X)$ splits into polynomials of degree $2^{i-3}$, so \eqref{c4-4} is the irreducible factorization. Combining \eqref{c4-3} and \eqref{c4-4}, we get
\[
f_{2^{k_1}}(X) = \prod_{\gamma \in \{\pm 1, \pm \alpha\}} (X + \gamma) \prod_{i=3}^{k_1} \prod_{u, v \in \mathbb{F}_3^*} (X^{2^{i-3}} + u\alpha + v), \quad \text{if } k_1 \geq 3.
\]

\textbf{Case 2: $k_1 = 0$, $k_2 \geq 1$.} Lemma \ref{lem2.2} gives
\begin{equation}\label{c4-5}
f_{5^{k_2}}(X) = (X - 1) \prod_{i=1}^{k_2} \Phi_{5^i}(X).
\end{equation}
The factorization of $\Phi_5(X)$ is
\[
\Phi_5(X) = \prod_{u \in \mathbb{F}_3^*} (X^2 + (u\alpha + 2)X + 1).
\]
For $1 \leq i \leq k_2$,
\begin{equation}\label{c4-6}
\Phi_{5^i}(X) = \Phi_5(X^{5^{i-1}}) = \prod_{u \in \mathbb{F}_3^*} (X^{2 \cdot 5^{i-1}} + (u\alpha + 2)X^{5^{i-1}} + 1).
\end{equation}
Lemmas \ref{lem2.2} and \ref{lem2.4} show $\Phi_{5^i}(X)$ factors into $\frac{\phi(5^i)}{2 \cdot 5^{i-1}}$ polynomials of degree $2 \cdot 5^{i-1}$, confirming that \eqref{c4-6} is the irreducible factorization over $\mathbb{F}_9$. Substituting into \eqref{c4-5} gives the factorization of $f_{5^{k_2}}(X)$.

\textbf{Case 3: $k_1 = 1$, $k_2 \geq 1$.} Lemma \ref{lem2.2} yields
\begin{equation}\label{c4-7}
f_{2 \cdot 5^{k_2}}(X) = (X - 1)(X + 1) \prod_{i=1}^{k_2} \Phi_{5^i}(X) \prod_{j=1}^{k_2} \Phi_{2 \cdot 5^j}(X).
\end{equation}
From Case 2, we have the irreducible factorization of $\Phi_{5^i}(X)$. Note that $\Phi_{10}(X)$ factors as
$$\Phi_{10}(X) = \prod_{u \in \mathbb{F}_3^*} (X^2 + (u\alpha + 1)X + 1).$$
By Lemma \ref{lem2.1}, for $1 \leq i \leq k_2$,
\begin{equation}\label{c4-8}
\Phi_{2 \cdot 5^i}(X) = \Phi_{10}(X^{5^{i-1}}) = \prod_{u \in \mathbb{F}_3^*} (X^{2 \cdot 5^{i-1}} + (u\alpha + 1)X^{5^{i-1}} + 1).
\end{equation}
Lemmas \ref{lem2.2} and \ref{lem2.4} confirm $\Phi_{2 \cdot 5^i}(X)$ factors into $\frac{\phi(2 \cdot 5^i)}{2 \cdot 5^{i-1}}$ polynomials of degree $2 \cdot 5^{i-1}$. It follows that \eqref{c4-8} is the irreducible factorization over $\mathbb{F}_9$. Substituting \eqref{c4-6} and \eqref{c4-8} into \eqref{c4-7} gives the factorization of $f_{2 \cdot 5^{k_2}}(X)$.

\textbf{Case 4: $2 \leq k_1 \leq 3$, $k_2 \geq 1$.} Similar to Case 3, we obtain:
\begin{align*}
f_{2^2 \cdot 5^{k_2}}(X) &= \prod_{\gamma \in \{\pm 1, \pm \alpha\}} (X + \gamma) \prod_{i=1}^{k_2} F_i(X), \\
f_{2^3 \cdot 5^{k_2}}(X) &= \prod_{\gamma \in \mathbb{F}_9^*} (X + \gamma) \prod_{i=1}^{k_2} F_i(X) G_i(X),
\end{align*}
where
\begin{align}\label{c4-9}
F_i(X) &= \prod_{u, v, w \in \mathbb{F}_3^*} (X^{2 \cdot 5^{i-1}} + (u\alpha + v)X^{5^{i-1}} + w), \\ 
\label{c4-10}
G_i(X) &= \prod_{\substack{u, v \in \mathbb{F}_3^* \\ t \in \{0, 1\}}} (X^{2 \cdot 5^{i-1}} + u\alpha^t X^{5^{i-1}} + v\alpha).
\end{align}

\textbf{Case 5: $k_1 \geq 4$, $k_2 \geq 1$.} Lemma \ref{lem2.2} gives
\begin{equation}\label{c4-11}
f_{2^{k_1} \cdot 5^{k_2}}(X) = f_{2^{k_1}}(X) \left( \prod_{j=0}^{3} \prod_{i=1}^{k_2} \Phi_{2^j \cdot 5^i}(X) \right) \left( \prod_{j=4}^{k_1} \prod_{i=1}^{k_2} \Phi_{2^j \cdot 5^i}(X) \right).
\end{equation}
From prior cases, we have factorizations for $f_{2^{k_1}}(X)$ and $\Phi_{2^j \cdot 5^i}(X)$ for $0 \leq j \leq 3$ and $1\le i\le k_2$. From Case 4, we know that $\Phi_{80}(X)$ factors as
$$\Phi_{80}(X) = \prod_{(\gamma_1, \gamma_2) \in S} (X^2 + \gamma_1 X + \gamma_2),$$
where $S = \{ (\pm 1, \pm \alpha - 1), (\pm (\alpha + 1), \alpha \pm 1), (\pm \alpha, \pm \alpha + 1), (\pm (\alpha - 1), -\alpha \pm 1) \}$. For $4 \leq j \leq k_1$, $1 \leq i \leq k_2$, Lemma \ref{lem2.1} gives
\begin{equation}\label{c4-12}
\Phi_{2^j \cdot 5^i}(X) = \Phi_{2^4 \cdot 5}(X^{2^{j-4} \cdot 5^{i-1}}) = \prod_{(\gamma_1, \gamma_2) \in S} (X^{2^{j-3} \cdot 5^{i-1}} + \gamma_1 X^{2^{j-4} \cdot 5^{i-1}} + \gamma_2).
\end{equation}
Lemmas \ref{lem2.2} and \ref{lem2.4} confirm $\Phi_{2^j \cdot 5^i}(X)$ factors into $\frac{\phi(2^j \cdot 5^i)}{2^{j-3} \cdot 5^{i-1}}$ irreducible polynomials of degree $2^{j-3} \cdot 5^{i-1}$, implying that \eqref{c4-12} is the irreducible factorization. Substituting \eqref{c4-12} into \eqref{c4-11} gives the irreducible factorization of $f_{2^{k_1} \cdot 5^{k_2}}(X)$:
\begin{align*}
X^{2^{k_1}\cdot5^{k_2}}-1&=\prod_{\gamma\in\{\pm 1,\pm\alpha\}}(X+\gamma)\prod_{j=3}^{k_1}E_j(X)\prod_{i=1}^{k_2}F_i(X)G_i(X)\prod_{k=4}^{k_1}H_{ik}(X),\ \text{if}\ k_1\ge 4,
\end{align*}
where
\begin{align*}
E_j(X)&=\prod_{u,v\in\fth^*}(X^{2^{j-3}}+u\alpha+v),\\
H_{ik}(X)&=\prod_{(\gamma_1,\gamma_2)\in S}(X^{2^{k-3}\cdot5^{i-1}}+\gamma_1X^{2^{k-4}\cdot 5^{i-1}}+\gamma_2),
\end{align*}
and where $F_i(X)$ and $G_i(X)$ are defined as \eqref{c4-9} and \eqref{c4-10}.
\end{exa}

\section*{Conflict of Interest}
The author declares that there is no conflict of interest.

\section*{acknowledgment}

This work was conducted during the author’s academic visit to RICAM, Austrian Academy of Sciences. The author sincerely thanks Professor Arne Winterhof at RICAM for his thorough review of the manuscript and for offering invaluable suggestions.

This research was partially supported by the China Scholarship Council Fund (Grant No.\ 202301010002) and the Scientific Research Innovation Team Project of China West Normal University (Grant No.\ KCXTD2024-7).

%

\end{document}